\documentclass[11pt]{article}
\usepackage[utf8]{inputenc}
\usepackage[T1]{fontenc}
\usepackage{comment}
\usepackage{authblk}
\usepackage{relsize}
\usepackage{changepage}
\usepackage{mathtools}
\usepackage[english]{babel}
\usepackage{amsmath}
\usepackage{amsthm}
\usepackage{amsfonts}
\usepackage{amssymb}
\usepackage{graphicx}
\usepackage[usenames,dvipsnames]{color}
\theoremstyle{plain}
\newtheorem{theorem}{Theorem}[section]
\newtheorem{corollary}[theorem]{Corollary}
\newtheorem{claim}[theorem]{Claim}
\newtheorem{remark}[theorem]{Remark}

\newtheorem{question}[theorem]{Question}
\newtheorem{lemma}[theorem]{Lemma}

\newtheorem{observation}[theorem]{Observation}

\makeatletter
\newcommand{\vast}{\bBigg@{4}}
\newcommand{\Vast}{\bBigg@{5}}
\makeatother
\definecolor{bulgarianrose}{rgb}{0.28, 0.02, 0.03}
\definecolor{gray}{rgb}{0.5, 0.5, 0.5}

\theoremstyle{definition}

\theoremstyle{remark}
\newtheorem*{fact*}{Fact}
\newtheorem*{question*}{Question}
\usepackage[left=2cm,right=2cm,top=2cm,bottom=2cm]{geometry}
\usepackage{amssymb}
\usepackage{hyperref}
\makeatletter
\def\namedlabel#1#2{\begingroup
    #2%
    \def\@currentlabel{#2}%
    \phantomsection\label{#1}\endgroup
}
\usepackage[normalem]{ulem}
\usepackage{dsfont}
\usepackage{accents}
\usepackage{verbatim}
\usepackage{ulem}
\usepackage{pgf,tikz,pgfplots}
\usepackage[all]{xy}

\usepackage{stackengine}
\stackMath
\newcommand\tsup[2][2]{%
 \def\useanchorwidth{T}%
  \ifnum#1>1%
    \stackon[-.5pt]{\tsup[\numexpr#1-1\relax]{#2}}{\scriptscriptstyle\sim}%
  \else%
    \stackon[.5pt]{#2}{\scriptscriptstyle\sim}%
  \fi%
}

\newcommand{\Bin}{\mathrm{Bin}}
\newcommand{\bard}{\overline{d}}

\usepackage{subfig}

\usepackage{enumerate}

\title{\scshape
  The giant component after percolation of product graphs}

\author{Lyuben Lichev}
\affil{Ecole Normale Sup\'erieure de Lyon, Lyon, France}

\begin{document}

\maketitle
 
\begin{abstract}
In this paper we show the existence of a sharp threshold for the appearance of a giant component after percolation of Cartesian products of graphs under assumptions on their maximum degrees and their isoperimetric constants. In particular, this generalises a work of Ajtai, Koml\'os and Szemer\'edi from 1982 concerning percolation of the hypercube in high dimension.
\end{abstract}

\hspace{1em}Keywords: giant component, product graph, percolation, random graph, sharp threshold

\hspace{1em}MSC Class: 05C76, 05C80

\section{Introduction}
The field of random graphs was born in a series of papers of Erd\H{o}s and Rényi~\cite{ER2, ER1, ER3}. The paper~\cite{ER1} concentrates in particular on the existence of a giant component in the random graphs $\mathcal G(n,M)$ and $\mathcal G(n, p)$, that is, a connected component that contains a constant proportion of all $n$ vertices in the graph. In the $\mathcal G(n,M)$ model, $M$ edges are chosen among all $\binom{n}{2}$ pairs of vertices uniformly at random to form a random graph with exactly $M$ edges, while in the $\mathcal G(n,p)$ model every pair of vertices forms an edge with probability $p$ in the final graph independently from all other pairs (or equivalently, $G\in \mathcal G(n,p)$ is a random subgraph of the complete graph on $n$ vertices after $p$-percolation of its edges). In~\cite{ER1}, Erd\H{o}s and Rényi proved the following (by now very classical) result: for any $\varepsilon > 0$, if $M\le (1-\varepsilon) n/2$, then all connected components in the random graph $G\in \mathcal G(n,M)$ have $O(\log n)$ vertices asymptotically almost surely (a.a.s.), while if $M\ge (1+\varepsilon) n/2$, then the largest component in the random graph $G\in \mathcal G(n,M)$ contains $\Omega(n)$ vertices but the second largest contains only $O(\log n)$ vertices a.a.s. Later Bollob{\'a}s~\cite{Bol1} and {\L}uczak~\cite{Luc1} made a precise analysis of the more complicated regime when $M = n/2 + o(n)$ and exhibited a critical window around $M = n/2$ of width of order $\Theta(n^{2/3})$, in which a number of connected components with $\Theta(n^{2/3})$ vertices in each happen to coexist a.a.s. All results above have natural analogues for $\mathcal G(n,p)$. Aldous~\cite{Ald} later made a beautiful connection between the sizes of the connected components in the critical regime in $\mathcal G(n,p)$ and the zeros of a Brownian motion with a suitable drift.

In fact, percolation of finite graphs was considered in many particular cases. Another classical example is the \emph{hypercube in dimension $n$}, denoted by $H_n$. The graph $H_n$ has vertices $\{0,1\}^n$ and two vertices $u$ and $v$ are connected by an edge if they differ in exactly one entry. In~\cite{ES} Erd\H{o}s and Spencer showed that if $p\le (1-\varepsilon)/n$, then a.a.s. $p$-percolation of $H_n$ leaves a graph with largest component, containing at most $o(2^n)$ vertices, and they conjectured that a component with $\Omega(2^n)$ vertices is a.a.s. present if $p\ge (1+\varepsilon)/n$. This conjecture was confirmed by Ajtai, Koml{\'o}s and Szemer{\'e}di in~\cite{AKS}. A following series of papers of Bollob{\'a}s, Kohayakawa and {\L}uczak~\cite{BKL}, Borgs, Chayes, van der Hofstad, Slade and Spencer~\cite{BCVdHSS1, BCVdHSS2, BCVdHSS3}, van der Hofstad and Nachmias~\cite{vdHN}, and Hulshof and Nachmias~\cite{HN} provides a deep understanding of the critical percolation on the hypercube in high dimension.

To the best of our knowledge there were two attempts for generalising the sharp threshold phenomenon for the existence of a giant component for large families of finite graphs. Chung, Horn and Lu~\cite{CHL} showed the existence of a sharp threshold under several conditions involving the spectrum of the adjacency matrix of the base graph. Sadly, their conditions are not satisfied for the hypercube $H_n$, see~\cite{Chu}. Alon, Benjamini and Stacey~\cite{ABC} proved the existence of a sharp threshold in expanders of uniformly bounded degree. Here as well, although the hypercube $H_n$ has indeed good expansion properties~\cite{Til}, its degree goes to infinity with $n$.

Our goal in this paper is to generalise the existence of a sharp threshold for the appearance of a giant component for Cartesian products of graphs under two assumptions: on the maximum degrees and on the isoperimetric constants of the graphs in the product. In particular, our result ensures the existence of a sharp threshold for the appearance of a giant component for the Cartesian product of any sequence of $n$ connected graphs with uniformly bounded orders, the hypercube $H_n$ being a particular case of the latter. We believe it is worth making the connection with Joos~\cite{Joo2}, who studies the threshold probability for connectivity of percolated sparse graphs and, as a corollary, completely solves the problem for Cartesian powers of a graph $G$.

\subsection{Notation and terminology}
For every positive integer $n$, we denote by $[n]$ the set $\{1,2,\dots,n\}$. In this paper, for any three positive real numbers $a,b,c$, by $a/bc$ or $a/b\cdot c$ we mean $a/(bc)$.

For a graph $G$, the \emph{order} of $G$ is the cardinality of its vertex set $V(G)$, and the \emph{size} of $G$ is the cardinality of its edge set $E(G)$. For a vertex $v\in V(G)$, we denote by $\deg_G(v)$, or just $\deg(v)$, the degree of $v$ in $G$, and by $CC_G(v)$, or just $CC(v)$, the connected component of $v$ in $G$. Then, the \emph{average degree} of $G$ is defined by
\begin{equation*}
\bard(G) = \dfrac{1}{|V(G)|}\sum_{v\in V(G)} \deg_G(v).
\end{equation*}
The maximum degree of a graph $G$ is denoted $\Delta(G)$, and the order of the largest connected component in $G$ is denoted $L_1(G)$. Finally, for any graph $G$ with $|V(G)|\ge 2$, the \emph{isoperimetric constant} of $G$ is given by
\begin{equation*}
    i(G) = \min_{\substack{S\subseteq V(G);\\ 1\le |S|\le |V(G)|/2}} \dfrac{|\partial S|}{|S|},
\end{equation*}
where $\partial S = \partial_G S$ is the set of edges in $G$ between a vertex in $S$ and a vertex in $V(G)\setminus S$. Clearly if $L_1(G) < |V(G)|$, then $i(G) = 0$. For a set $S\subseteq V(G)$, we also denote by $N_G(S)$, or just $N(S)$, the set of vertices in $G$ at graph distance 1 from $S$ in $G$, and also $N_G[S]$, or simply $N[S]$, is defined as $S\cup N(S)$.

For any sequence of $n$ graphs $G_1, G_2, \dots, G_n$, the \emph{Cartesian product of $G_1, G_2,\dots, G_n$}, denoted by $G_1\square G_2\square \dots \square G_n$ or $\square_{i\in [n]} G_i$, is the graph with vertex set
\begin{equation*}
    \{(v_1, v_2, \dots, v_n)\hspace{0.2em}|\hspace{0.2em} \forall i\in [n], v_i\in V(G_i)\}
\end{equation*}
and edge set
\begin{equation*}
    \{(u_1, u_2, \dots, u_n) (v_1, v_2, \dots, v_n)\hspace{0.2em}|\hspace{0.2em} \exists i\in [n], \forall j\neq i, u_j=v_j\text{ and } u_iv_i\in E(G_i)\}.
\end{equation*}

A $p$-percolation of a graph $G$ is a random process in which every edge in $G$ is retained with probability $p$ and deleted with probability $1-p$, independently from all other edges. If an edge is retained, we say that it is \emph{open}, and if it is deleted, we say that it is \emph{closed}. The graph consisting of all open edges is a random subgraph of $G$, which we denote by $G_p$. 

For a sequence of probability spaces $(\Omega_n, \mathcal F_n, \mathbb P_n)_{n\geq 1}$ and a sequence of events $(A_n)_{n\geq 1}$, where $A_n\in \mathcal F_n$ for every $n\geq 1$, we say that $(A_n)_{n\geq 1}$ happens \emph{asymptotically almost surely} or \emph{a.a.s.} if $\underset{n\to +\infty}{\lim}\mathbb P_n(A_n) = 1$. The sequence of events $(A_n)_{n\geq 1}$ itself is said to be \emph{asymptotically almost sure} or again \emph{a.a.s}. 

Our main result, Theorem~\ref{main thm}, is of asymptotic nature. Therefore, below we use the well-known asymptotic notations $o, O, \Omega$ and $\Theta$. For two functions $f,g:\mathbb N\to \mathbb R^+$ we also write $f(n)\ll g(n)$ or $g(n)\gg f(n)$ if $f(n) = o(g(n))$. Moreover, if the limit variable is not $n$, we will indicate this using lower indices such as $O_x$.

\subsection{Our result}
Throughout the paper we fix two absolute constants $\gamma > 0$ and $C\in \mathbb N$ (that is, these constants do not depend on any other parameters in the sequel). Let $(G_{n,j})_{n\in \mathbb N, j\in [n]}$ be finite connected graphs with at least one edge such that, for every $n\in \mathbb N$ and $j\in [n]$:
\begin{enumerate}
    \item\label{cn 1} $\Delta(G_{n,j})\le C$, and
    \item\label{cn 2} $i(G_{n,j})\ge n^{-\gamma}$.
\end{enumerate}
Define $G_{[n]} = \square_{j\in [n]} G_{n,j}$. In the sequel we write $G_j$ for $G_{n,j}$, $G$ for $G_{[n]}$, and $\bard$ for $\bard(G_{[n]})$, hopefully taking enough care to ensure that no confusion arises due to this abuse of notation. We insist that we reserve the notation $G_p$ for the subgraph of $G$ after $p$-percolation, so $G_p$ is not part of $(G_j)_{j\in [n]}$. For any vertex $v\in V(G)$ we denote by $CC_p(v) = CC_{G_p}(v)$. Since a vertex $v = (v_1, \dots, v_n)\in V(G)$ has degree $\sum_{j=1}^n
\deg_{G_j}(v_j)$, we conclude that 
\begin{equation}\label{eq 1}
    \bard = \sum_{j=1}^n \bard(G_j).
\end{equation}
Now we present the main result of the paper.

\begin{theorem}\label{main thm}
Fix $\varepsilon\in (0,1)$. In the above setup:
\begin{enumerate}
    \item[a)] if $p = (1-\varepsilon)/\bard$, then a.a.s. $L_1(G_p)\le \exp\left(-\dfrac{\varepsilon^2 n}{9 C^2}\right) |V(G)|$, and
    \item[b)] if $p = (1+\varepsilon)/\bard$, then there is a positive constant $c_1 = c_1(\varepsilon, \gamma, C)$ such that a.a.s. $L_1(G_p)\ge c_1 |V(G)|$.
\end{enumerate}
\end{theorem}

\begin{remark}\label{rem 2.1}
One may replace the constant $C$ with a function $n^{\alpha}$ in condition~\ref{cn 1}, where $\alpha = \alpha(\gamma)$ is a positive constant, and Theorem~\ref{main thm} will still be valid. We present the proof only of the given more simplified version of Theorem~\ref{main thm} for two reasons: first, the idea of the proof is the same and this more general version would only make the exposition more technical, and second, we believe that even this more general framework does not fully explain the existence of a sharp threshold for the giant component problem for product graphs.
\end{remark}

Let us give a quick overview of the proof of Theorem~\ref{main thm}. The first point concerns the study of two subcritical exploration processes. The first one ensures an upper bound on the order of the union of all components, containing at least one vertex of ``high'' degree. The second process deals with the remaining ``low'' degree vertices conditionally on the edges, exposed during the first process, and is therefore directly dominated by a subcritical branching process. The proof of the second point is inspired by the special case of the hypercube, studied in~\cite{AKS}. It relies on consecutively constructing connected components with larger and larger polynomial orders via the technique of two-round exposure (or rather multi-round exposure in our case). Once the correct polynomial order is attained, we show by the same technique that the condition~\ref{cn 2} on the isoperimetric constants of the graphs in the product ensures that a constant proportion of the above components are merged together in $G_p$ a.a.s.

The paper is organised as follows. In Section~\ref{sec prelims} we introduce several preliminary results. Then, in Section~\ref{sec pt 1 proof} we prove Theorem~\ref{main thm}~a), and in Section~\ref{sec pt 2 proof} we prove Theorem~\ref{main thm}~b). Finally, Section~\ref{sec discussion} is dedicated to a discussion and a couple of open questions.

\section{Preliminaries}\label{sec prelims}

\subsection{Probabilistic preliminaries}

\textbf{Chernoff's inequality:} We first state a version of the famous Chernoff's inequality, see e.g. (\cite{JLR}, Theorem 2.1).

\begin{lemma}[\cite{JLR}, Theorem 2.1]\label{chernoff}
Let $X \sim \mathrm{Bin}(n,p)$ be a Binomial random variable with parameters $n$ and $p$. Then, for any $t \ge 0$ we have
\begin{eqnarray*}
\mathbb P( X \ge \mathbb E[X] + t ) &\le& \exp \left( - \frac {t^2}{2 (\mathbb E[X] + t/3)} \right), \text{ and}\\
\mathbb P( X \le \mathbb E[X] - t ) &\le& \exp \left( - \frac {t^2}{2 \mathbb E[X]} \right).
\end{eqnarray*}
\end{lemma}

\vspace{1em}

\noindent
\textbf{The bounded difference inequality:} The following well-known inequality is a simple consequence of the Azuma-Hoeffding martingale inequality, see~\cite{JLR} or also~\cite{Mac} for an improvement.

\begin{theorem}[The bounded difference inequality, see e.g.~\cite{JLR}]\label{azuma}
Consider a sequence $(X_i)_{i\in [n]}\in \prod_{i\in [n]} \Lambda_i$ of $n$ independent random variables. Fix a function $f:\prod_{i\in [n]} \Lambda_i\to \mathbb R$ and suppose that there exist $(C_i)_{i\in [n]}$ such that, for every $i\in [n]$, $(x_j)_{j\in [n]}\in \prod_{j\in [n]} \Lambda_j$ and $x'_i\in \Lambda_i$, we have
\begin{equation*}
    |f(x_1, \dots, x_{i-1}, x_i, x_{i+1}, \dots, x_n) - f(x_1, \dots, x_{i-1}, x'_i, x_{i+1}, \dots, x_n)|\le C_i.
\end{equation*}
Then, for every $t\ge 0$,
\begin{equation*}
\mathbb P(|f(X_1,\dots,X_n) - \mathbb E[f(X_1,\dots,X_n)]|\ge t)\le 2\exp\left(- \dfrac{t^2}{2\sum_{i\in [n]} C^2_i}\right).
\end{equation*}
\end{theorem}

\vspace{1em}

\noindent
\textbf{The Bienaym\'e-Galton-Watson random tree:} By now a very well-known and studied model is the \emph{Bienaym\'e-Galton-Watson random tree}, or \emph{BGW tree}. Let $\nu$ be a probability distribution over $\mathbb N\cup \{0\}$ and let $X$ be a random variable with distribution $\nu$. The BGW tree with progeny distribution $\nu$ is constructed as follows. Starting from a vertex $v_0$ (the root), every vertex gives birth (just once) to a random number of children, distributed according to $\nu$ and independent from all other vertices in the tree. The BGW tree is \emph{subcritical} if $\mathbb E[X] < 1$, \emph{supercritical} if $\mathbb E[X] > 1$, and \emph{critical} otherwise. It is a basic fact in the theory of branching processes that a subcritical BGW tree is almost surely finite while a supercritical BGW tree has strictly positive probability to be infinite. The next lemma makes the first statement more precise by giving a probabilistic estimate on the size of a subcritical BGW tree, see e.g.~(\cite{Bla}, Theorem 2.3.2).

\begin{lemma}[\cite{Bla}, Theorem 2.3.2]\label{size lem}
Let $T$ be a subcritical BGW tree such that $\mathbb E[s^X] < +\infty$ for some $s > 1$. Define
\begin{equation*}
    h_{\nu} = \sup_{\theta > 0} \left(\theta - \log\left(\mathbb E[\exp(\theta X)]\right)\right).
\end{equation*}
Then, for every $k\ge 1$ we have
\begin{equation*}
    \mathbb P(|V(T)|\ge k)\le \exp(-kh_{\nu}).
\end{equation*}
\end{lemma}

Fix $\varepsilon\in (0,1)$. We will use the above result in the particular case when $X\sim \Bin(n, p)$ with $p = (1-\varepsilon)/n$. We have
\begin{equation*}
    \mathbb E[\exp(\theta X)] = \sum_{i=0}^n \binom{n}{k} \exp(\theta k) p^k (1-p)^{n-k} = (1-p+\exp(\theta) p)^n.
\end{equation*}
Then, using that $\log(1+\theta) = \theta + O_{\theta}(\theta^2)$, as $n\to +\infty$ we have
\begin{align*}
    h_X =\hspace{0.3em} &\sup_{\theta > 0} (\theta - n\log(1-p+\exp(\theta) p))\\ 
    =\hspace{0.3em} &\sup_{\theta > 0} (\theta - n(-p+\exp(\theta) p + O(1/n^2)))\\
    =\hspace{0.3em} &\sup_{\theta > 0} (\theta + 1 - \varepsilon - (1-\varepsilon)\exp(\theta) + O(1/n)).
\end{align*}

Since $\exp(\theta) = 1+\theta+O_{\theta}(\theta^2)$, the latter quantity tends to a constant $\phi = \phi(\varepsilon) > 0$ as $n\to +\infty$, where
\begin{equation}\label{eq phi}
    \phi(\varepsilon) = \sup_{\theta > 0} (\theta + 1 - \varepsilon - (1-\varepsilon)\exp(\theta)).
\end{equation}

\begin{corollary}\label{size cor}
The BGW tree $T$ with progeny distribution $\Bin(n, (1-\varepsilon)/n)$ satisfies
\begin{equation*}
    \mathbb P(|V(T)|\ge k)\le \exp(-(1+o(1))k\phi).
\end{equation*}
In particular, for every $\varepsilon \in (0,1)$ there is $k_0 = k_0(\varepsilon)\in \mathbb N$ such that for every $k\ge k_0$ we have
\begin{equation*}
    \mathbb P(|V(T)|\ge k)\le \exp(-k\phi/2).
\end{equation*}
\end{corollary}

\subsection{Combinatorial preliminaries}

\textbf{The isoperimetric constant of a product graph:} Recall that $G$ is a graph, defined as a Cartesian product of the graphs $G_1, G_2, \dots, G_n$. The next result, due to Chung and Tetali~\cite{CT}, makes a connection between the isoperimetric constant of $G$ and the isoperimetric constants of $(G_k)_{k\in [n]}$, see also Tillich~\cite{Til} for a slight improvement.

\begin{theorem}[\cite{CT}, Theorem 2]\label{isoper thm}
\begin{equation*}
    \frac{1}{2} \min_{k\in [n]} i(G_k)\le i(G)\le \min_{k\in [n]} i(G_k).
\end{equation*}
\end{theorem}

We directly deduce the following corollary.

\begin{corollary}\label{isoper cor}
Under condition~\ref{cn 2} on $(G_k)_{k\in [n]}$ we have $n^{-\gamma}/2\le i(G)$.
\end{corollary}

\vspace{1em}

\noindent
\textbf{The largest connected component and ``balanced'' empty cuts:} The following easy observation makes a connection between empty cuts in a graph and the size of the largest connected component.

\begin{observation}\label{ob discrete cont}
Fix $k\in \mathbb N$. Let $H$ be a graph with $h$ vertices and let $\mathcal C_1, \dots, \mathcal C_k$ be disjoint connected subgraphs of $H$ such that $\cup_{j\in [k]} V(\mathcal C_j) = V(H)$. Suppose that for any set $J\subseteq [k]$ such that $|\cup_{j\in J} V(\mathcal C_j)|\in [h/3, 2h/3]$, there exists an edge in $H$ between a vertex in $\cup_{j\in J} \mathcal C_j$ and a vertex in $\cup_{j\in [k]\setminus J} \mathcal C_j$. Then, there is a connected component of $H$ that contains more than $h/3$ vertices.
\end{observation}
\begin{proof}
We argue by contradiction. Suppose that $H$ contains $m$ connected components and each of them has order at most $h/3$. Then, consider the graphs $H_0 = \varnothing, H_1, H_2, \dots, H_m$, where for every $i\in [m]$ we define $H_{i}$ to be the union of $H_{i-1}$ and some connected component in the graph $H\setminus H_{i-1}$. Since $|V(H_m)| = h$ and for every $i\in [m]$, $|V(H_i)| - |V(H_{i-1})|\le h/3$, by discrete continuity there is $\ell\in [m]$ such that $|V(H_{\ell})|\in [h/3, 2h/3]$. This is a contradiction with the assumption in the statement for the family $\{\mathcal C_j: \mathcal C_j\cap H_{\ell} = \emptyset\}$. The observation is proved.
\end{proof}

\section{\texorpdfstring{Proof of Theorem~\ref{main thm}~a) -- the subcritical regime}{}}\label{sec pt 1 proof}

We begin with a proof of Theorem~\ref{main thm}~a). Our first step will be to estimate the number of vertices of $G$ of degree at least $(1+\varepsilon/2)\bard$. For every $i\in [n]$, let $X_i$ be the degree of a uniformly chosen vertex in $G_i$. For every $i\in [n]$, define $S_i = X_1+X_2+\dots+X_i$. Since $\mathbb E[S_n] = \bard$ by \eqref{eq 1} and for every $i\in [n]$ we have $\Delta(G_i)\le C$, we conclude by the bounded difference inequality (Theorem~\ref{azuma}) that
\begin{equation}\label{eq 2}
    \mathbb P(|S_n - \bard|\ge \varepsilon \bard/2)\le 2\exp\left(-\dfrac{\varepsilon^2 \bard^2/4}{2C^2 n}\right)\le 2\exp\left(-\dfrac{\varepsilon^2 n}{8C^2}\right).
\end{equation}
The second inequality comes from the fact that $\bard\ge n$: indeed, by \eqref{eq 1} we have $\bard = \sum_{i=1}^n \bard(G_i)$, and every graph among $(G_i)_{i\in [n]}$ is connected and therefore its average degree is at least 1. Therefore, the number of vertices in $G$ with degree more that $(1+\varepsilon/2)\bard$ is at most a $2\exp\left(-\frac{\varepsilon^2 n}{8C^2}\right)$-proportion of all vertices of $G$.

\begin{proof}[Proof of Theorem~\ref{main thm}~a)]
First, we prove that the number of vertices connected via a path in $G_p$ to a vertex of degree at least $(1+\varepsilon/2) \bard$ in $G$ is at most $\exp\left(-\dfrac{\varepsilon^2 n}{9 C^2}\right) |V(G)|$ a.a.s. Indeed, let $U$ be the set of vertices of degree at least $(1+\varepsilon/2) \bard$ in $G$. Then, by \eqref{eq 2} we have $|U|\le \exp\left(-\dfrac{\varepsilon^2 n}{8 C^2}\right) |V(G)|$. We consider the following stochastic process. Let $U_0 = N_{G_p}(U)$, and for every positive integer $k$ we inductively define $U_k = N_{G_p}(U_{k-1})\setminus (U\cup U_0\cup \dots \cup U_{k-1})$. Since for every set $V\subseteq V(G)\setminus U$ we have $|\partial V|\le (1+\varepsilon/2)\bard |V|$, we have that, for every $k\ge 1$, 
\begin{equation*}
    \mathbb E[|U_k|\hspace{0.2em}|\hspace{0.2em}U_{k-1}]\le p|\partial U_{k-1}|\le p(1+\varepsilon/2)\bard |U_{k-1}|\le (1-\varepsilon/2)|U_{k-1}|. 
\end{equation*}

We conclude that for every $k\ge 1$, 
\begin{equation*}
\mathbb E[|U_k|]\le (1-\varepsilon/2)^k \mathbb E[|U_0|]\le (1-\varepsilon/2)^k Cn |U|.
\end{equation*}
Thus, we get by Markov's inequality that
\begin{equation*}
\mathbb P(\exists k\ge 1, |U_k|\ge (1-\varepsilon/2)^{k/2} Cn^2 |U|)\le \sum_{k\ge 1} \mathbb P(|U_k|\ge (1-\varepsilon/2)^{k/2} Cn^2 |U|)\le \sum_{k\ge 1} \dfrac{(1-\varepsilon/2)^{k/2}}{n} = o(1).
\end{equation*}
We deduce that the union of all connected components of $G_p$, containing at least one vertex of $U$, contains at most $\sum_{k\ge 1} (1-\varepsilon)^{k/2} Cn^2 |U| = \Theta(n^2 |U|) = o(|V(G)|)$ vertices a.a.s.

Denote the set of vertices in all explored connected components by $\overline{U}$. Here, an edge of $G$ is explored if the fact that it is open or not has been revealed, and a connected component is explored if all edges it contains have been explored and are open, while all edges on its boundary have been explored and are closed. After exploring all connected components of $G_p$, containing at least one vertex in $U$, we are left with unexplored edges, incident only to vertices of degree less than $(1+\varepsilon/2) \bard$ in $G$. We prove that in the remainder of $G_p$ there is a.a.s. no connected component of order more than $\lceil 4\log |V(G)|/\phi(\varepsilon/2)\rceil$, with $\phi$ defined in \eqref{eq phi}. Indeed, choose any vertex $v$ and start an exploration process of its connected component $CC_p(v)$ in $G_p$. Note that any vertex in $V(G)\setminus \overline{U}$ is incident to less than $(1+\varepsilon/2) \bard$ unexplored edges. Thus, the number of edges in $CC_p(v)$ is stochastically dominated by the number of explored edges in a BGW tree $T$ with progeny distribution $\Bin(\lfloor(1+\varepsilon/2)\bard\rfloor, p)$. Since $\bard\to +\infty$ with $n$ and 
\begin{equation*}
p = \dfrac{1-\varepsilon}{\bard}\le \dfrac{1-\varepsilon/2}{(1+\varepsilon/2)\bard}\le \dfrac{1-\varepsilon/2}{\lfloor(1+\varepsilon/2)\bard\rfloor},
\end{equation*}
by Corollary~\ref{size cor} we get that for every $k\ge 1$ and for every $n$ large enough
\begin{equation*}
\mathbb P(|V(CC_p(v))|\ge k)\le \mathbb P(|V(T)|\ge k)\le \exp\left(-\dfrac{k\phi(\varepsilon/2)}{2}\right).
\end{equation*}
Choosing $k = k_0 := \lceil 4\log |V(G)|/\phi(\varepsilon/2)\rceil$, we get that with probability at most $1/|V(G)|^2$, $CC_p(v)$ contains at least $k_0$ vertices. A union bound over all vertices in $V(G)\setminus \overline{U}$ implies that with probability at most $1/|V(G)|$, the largest component in $G_p$, containing no vertex in $U$, is of order at least $k_0$. Since
\begin{equation*}
    \exp\left(-\dfrac{\varepsilon^2 n}{9 C^2}\right) |V(G)|\ge 2^{-n/2} |V(G)|\ge \sqrt{|V(G)|}\gg \log |V(G)|,
\end{equation*}
the proof is finished.
\end{proof}

\section{\texorpdfstring{Proof of Theorem~\ref{main thm}~b) -- the supercritical regime}{}}\label{sec pt 2 proof}

The remainder of the paper will be directed towards proving Theorem~\ref{main thm}~b). The main technique, well-known under the name \emph{two-round exposure} or \emph{sprinkling}, has by now become a classical tool in the field of random graphs. It states that the graph $G_p$ may be realised as a union of two random graphs on the same vertex set $G_{p_1}$ and $G_{p_2}$, sampled independently from each other, where $(1-p_1)(1-p_2) = 1-p$. Indeed, the probability that an edge in $G$ does not appear in $G_p$ is $1-p$, while by independence the probability that an edge in $G$ does not appear in $G_{p_1}\cup G_{p_2}$ is $(1-p_1)(1-p_2)$. Moreover, in both $G_p$ and $G_{p_1}\cup G_{p_2}$, different edges appear independently from each other.

In our case, inspired by \cite{AKS}, we show that one may choose $p_1$ and $p_2$ appropriately so that a.a.s. $G_{p_1}$ consists of a number of connected components of order at least $\Omega(n^k)$ for some large enough positive integer $k$, which contain a constant proportion of all vertices of $G$. Then, at the second stage, we show that a.a.s. a constant proportion of all such component merge in a connected component of size $\Theta(|V(G)|)$.

In the sequel, $\deg(v)$ will refer to the degree of a vertex $v$ in $G$.

\begin{observation}\label{ob neighbours}
For some $i\in [n]$, let 
\begin{equation*}
v_1 = (w_1, \dots, w_{i-1}, u_1, w_{i+1}, \dots, w_n)\text{ and }v_2 = (w_1, \dots, w_{i-1}, u_2, w_{i+1}, \dots, w_n)
\end{equation*}
be two vertices in $G$. Then, $|\deg(v_1) - \deg(v_2)|\le C-1$.
\end{observation}
\begin{proof}
We have $\deg(v_1) - \deg(v_2) = \deg_{G_i}(u_1) - \deg_{G_i}(u_2)$. The claim follows since the graph $G_i$ is connected and has maximum degree at most $C$.
\end{proof}

\begin{corollary}\label{cor neighbours}
If $C\ge 2$, two vertices $v_1$ and $v_2$ in $G$ are at graph distance at least $\dfrac{|\deg(v_1) - \deg(v_2)|}{C-1}$.
\end{corollary}

Let $D$ be the set of vertices of degree at most $(1-\varepsilon/2)\bard$ in $G$.

\begin{lemma}\label{AKS lem 1}
Fix $\varepsilon\in (0, 0.1)$ and $p\ge  (1+7\varepsilon/8)/\bard$. There is a constant $c_1 = c_1(\varepsilon) > 0$ such that, for every large enough $n$, every vertex in $G$ of degree at least $(1-\varepsilon/4) \bard$ participates in a connected component of $G_p$ of size at least $\varepsilon \bard/4C$ with probability at least $c_1$.
\end{lemma}
\begin{proof}
Fix a vertex $v_0$ satisfying $\deg_G(v_0)\ge (1-\varepsilon/4) \bard$ and start an exploration process of the connected component of $v_0$ in $G_p$ as follows. We divide the vertices of $G_p$ in several categories: \emph{active}, when the edges, incident to this vertex, have not been explored from the vertex itself, but it has been attained via a path from $v_0$, \emph{passive}, if the vertex was active before but the edges in its neighbourhood have been explored from it, and \emph{processed}, when the vertex is either active or passive. For example, in the beginning only the vertex $v_0$ is active and there are no passive vertices. A reformulation of the statement of the lemma is that, by starting an exploration process of $G_p$ from $v_0$, with probability at least $c_1$ at least $\varepsilon \bard/4C$ vertices will be processed in the end. Start by exploring all edges in $G$, going out of $v_0$, and make all neighbours of $v_0$ in $G_p$ active. Then, make $v_0$ passive and find an active vertex $v_1$, if it exists. Then, explore all edges incident to $v_1$ in $G$ and make all neighbours of $v_1$ in $G_p$ that have not yet been processed active. Then, make $v_1$ passive and find an active vertex $v_2$, if it exists, etc. Continue with the exploration until either all or at least $\varepsilon \bard/4C$ vertices in the connected component of $v_0$ in $G_p$ have been processed.\par

Fix an integer $n\ge 8C^2/\varepsilon$. If $C\ge 2$, by Corollary~\ref{cor neighbours} every vertex of degree at least $(1-\varepsilon/4) \bard$ in $G$ is at distance at least $\varepsilon \bard/4(C-1)\ge \varepsilon \bard/4C + 1$ from $D$. The same holds if $C=1$ since $D = \emptyset$ then. Fix any integer $k\in [2,\varepsilon \bard/4C]$ (this interval is non-empty since $\bard\ge n\ge 8C/\varepsilon$). Under the assumption that at most $k$ vertices in $G_p$ have been made passive before exploring the neighbourhood of a particular active vertex $u$, at least 
\begin{equation*}
\deg_G(u) - 1 - C - (C-1) (k-1)\ge \deg_G(u) - Ck\ge (1-\varepsilon/2) \bard - \varepsilon \bard/4\ge (1 - 3\varepsilon/4) \bard  
\end{equation*}
neighbours of $u$ have never been processed before (here, $1+C+(C-1)(k-1)\le Ck$ is an upper bound of the total number of processed vertices after $k$ steps, and $\deg_G(u)\ge (1-\varepsilon/2) \bard$ since $u$ is at distance at most $k\le \varepsilon \bard/4C$ from $v_0$ that satisfies $\deg_G(v_0)\ge (1-\varepsilon/4) \bard$). Therefore, until the number of processed vertices is at most $\varepsilon \bard/4C$, the number of edges of $G$, incident to the currently explored vertex $u$ and leading to vertices which have never been processed before, is at least $(1 - 3\varepsilon/4) \bard$. We may conclude that the exploration of the connected component of $v_0$ in $G_p$, up to the moment of finding $\lceil \varepsilon \bard/4C\rceil$ processed vertices, stochastically dominates the exploration of a BGW tree with progeny distribution $\Bin\left(\lceil(1-3\varepsilon/4)\bard\rceil, \frac{1+7\varepsilon/8}{\bard}\right)$. For every $\varepsilon\le 0.1$, the BGW tree with these parameters is supercritical since
\begin{equation*}
    \lceil(1-3\varepsilon/4)\bard\rceil\cdot \dfrac{1+7\varepsilon/8}{\bard}\ge 1 + \dfrac{\varepsilon}{8} - \dfrac{21\varepsilon^2}{32} > 1,
\end{equation*}
and therefore it has probability $c_1 = c_1(\varepsilon) > 0$ to grow to infinity. Thus, with probability at least $c_1$, the exploration of $G_p$ from $v_0$ leads to at least $\varepsilon \bard/4C$ processed vertices, which proves the lemma.
\end{proof}

Following~\cite{AKS}, we call a connected subgraph of $G_p$ a \emph{cell}. Note that a connected component of $G_p$ is a cell, but a cell does not have to be a connected component of $G_p$ itself. Fix the constant $c_1 = c_1(\varepsilon)$ given by Lemma~\ref{AKS lem 1}. We say that a vertex $v$ is a neighbour of a set of vertices $A$ in a graph $H$ if there is a vertex $u\in A$, which is a neighbour of $v$ in $H$. Let $P_p$ be the following property of a vertex $v$ of $G$:\\

\begin{center}
the vertex $v$ is a neighbour in $G$ to at least $c_1 \varepsilon n/64 C$ disjoint cells in $G_p$, each of order at least $\varepsilon n/8C$.
\end{center}

\begin{lemma}\label{AKS lem 2}
Fix $\varepsilon\in (0, 0.1)$ and $p = (1+\varepsilon)/\bard$. There is a constant $c_2 = c_2(\varepsilon) > 0$ such that, for every large enough $n$, every vertex in $G$ of degree at least $(1-\varepsilon/8) \bard$ has property $P_p$ with probability at least $1 - \exp(- c_2 n)$.
\end{lemma}
\begin{proof}
Fix a vertex $v$ satisfying $\deg_G(v)\ge (1-\varepsilon/8) \bard$. Let $v = (v_1, v_2, \dots, v_n)$. For every $i\in [n]$, let $u_i$ be a neighbour of $v_i$ in $G_i$, and define 
\begin{equation*}
    H_i = \left(\underset{j\in [i-1]}{\square} v_j\right)\square\hspace{0.2em} u_i\hspace{0.2em}\square\left(\underset{k\in [n]\setminus [i]}{\square} G_k\right).
\end{equation*}
Thus, the graph $H_i$ is a conveniently chosen projection of $G$, which is isomorphic to $\underset{k\in [n]\setminus [i]}{\square} G_k$. Also, for every $i\in [n]$, let $\hat{v}_i = (v_1, \dots, v_{i-1}, u_i, v_{i+1}, \dots, v_n)$. Then, we claim that, for every $i\le i_{\max} := \lfloor\varepsilon n/16C\rfloor - 1$, the vertex $\hat{v}_i$ has degree at least $(1-\varepsilon/8) \bard - C(i+1)\ge (1-\varepsilon/4) \bard(H_i)$ in $H_i$. Indeed, 
\begin{align*}
    \left(1-\dfrac{\varepsilon}{8}\right)\bard - C(i+1) & \ge\hspace{0.3em} \left(1-\dfrac{\varepsilon}{8}\right)\bard(H_i) - C(i+1)\\ 
    & \ge\hspace{0.3em} \left(1-\dfrac{\varepsilon}{8}\right)\bard(H_i) - \dfrac{\varepsilon n}{16}\\
    & \ge\hspace{0.3em} \left(1-\dfrac{\varepsilon}{4}\right)\bard(H_i) + \dfrac{\varepsilon}{8}\left(\bard(H_i) - \dfrac{n}{2}\right)\\
    & \ge\hspace{0.3em} \left(1-\dfrac{\varepsilon}{4}\right)\bard(H_i) + \dfrac{\varepsilon}{8}\left(n- i - \dfrac{n}{2}\right)\ge \left(1-\dfrac{\varepsilon}{4}\right)\bard(H_i).
\end{align*}

Moreover, by the choice of $i$ we also have
\begin{equation*}
    (1+\varepsilon) \bard(H_i)\ge (1+\varepsilon) (\bard - C(i+1))\ge (1+\varepsilon)\bard - 2C(i+1)\ge \left(1+\dfrac{7 \varepsilon}{8}\right) \bard.
\end{equation*}
Thus, $p = (1+\varepsilon)/\bard\ge (1+7\varepsilon/8)/\bard(H_i)$, and we may apply Lemma~\ref{AKS lem 1} to the vertex $\hat{v}_i$ in $H_i$ and deduce that the probability that $\hat{v}_i$ participates in a cell in $H_i$ of order at least $\varepsilon \bard(H_i)/4C\ge \varepsilon n/8C$ is at least $c_1$. Since the graphs $(H_i)_{i\in [i_{\max}]}$ are disjoint, the events 
\begin{equation*}
    \left(A_i := \{\text{the connected component of } \hat{v}_i \text{ in } H_i \text{ contains at least } \varepsilon n/8C \text{ vertices}\}\right)_{i\in [i_{\max}]}
\end{equation*}
are independent and each of them happens with probability at least $c_1$. Thus, by Chernoff's inequality (Lemma~\ref{chernoff}) for $(\mathds 1_{A_i})_{i\in i_{\max}}$ with $t = \mathbb E[\sum_{i=1}^{i_{\max}} \mathds 1_{A_i}]/2\ge c_1 i_{\max}/2$, the vertex $v$ is incident to at least $c_1 i_{\max}/2\ge c_1 (\varepsilon n/32 C)/2 = c_1 \varepsilon n/64 C$ disjoint cells in its neighbourhood in $G$ with probability at least $1-\exp(-t/8)\ge 1 - \exp(-c_1 \varepsilon n/512 C)$. Thus, $c_2 = c_1 \varepsilon / 512 C$ satisfies our requirements, and the lemma is proved.
\end{proof}

Fix the constant $c_2 = c_2(\varepsilon) > 0$, given by Lemma~\ref{AKS lem 2}.

\begin{corollary}\label{AKM cor 2 bar}
Fix $\varepsilon\in (0,0.1)$ and $p = (1+\varepsilon)/\bard$. Then, every vertex $v\in V(G)$ with $\deg_G(v)\ge (1-\varepsilon/16)\bard$ satisfies the following property with probability at least $1 - \exp(-(c_2+o(1))n)$: $v$ is a neighbour in $G$ to at least $c_1\varepsilon n/64 C$ disjoint cells of $G_p$, each containing at least $\varepsilon n/17 C$ vertices with the property $P_p$.
\end{corollary}
\begin{proof}
Fix any vertex $v$ satisfying $\deg_G(v)\ge (1-\varepsilon/16)\bard$. By Lemma~\ref{AKS lem 2} it has probability at least $1-\exp(-c_2n)$ to have property $P_p$. We condition on this event. Then, for every cell $\mathcal C$ among the first $\lceil c_1\varepsilon n/64 C\rceil$ disjoint neighbouring cells of size at least $\varepsilon n/8 C$, corresponding to $v$, put a label $\ell_v$ on the $\lceil\varepsilon n/17 C\rceil$ vertices of $\mathcal C$ that are closest to $v$ in the graph $G_p$ (if some set of vertices is at the same distance to $v$ in $G_p$, make an arbitrary choice which of them to label, if necessary). Thus, for every vertex $u$ which has received a label $\ell_v$ we have $d_G(u,v)\le d_{G_p}(u,v)\le \lceil\varepsilon n/17 C\rceil$. Moreover, by Corollary~\ref{cor neighbours} for every large enough $n$ we have $|\deg(u) - \deg(v)|\le C\lceil\varepsilon n/17 C\rceil\le \varepsilon n/16$ and so $\deg(u)\ge \deg(v) - \varepsilon n/16\ge \bard - \varepsilon \bard/16 - \varepsilon n/16 \ge (1 - \varepsilon/8) \bard$.

Note that a total of at most $\lceil c_1\varepsilon n/64 C\rceil\cdot (\varepsilon n/16 C) = \Theta(n^2)$ vertices will receive the label $\ell_v$, and furthermore by Lemma~\ref{AKS lem 2} each of these vertices has property $P_p$ with probability at least $1 - \exp(-c_2 n)$. Then, conditionally on the event that $v$ has property $P_p$, any vertex $u$ with label $\ell_v$ has property $P_p$ with probability
\begin{align*}
    & \mathbb P(u \text{ has }P_p\hspace{0.2em}|\hspace{0.2em} v\text{ has } P_p) = \dfrac{\mathbb P(u \text{ and } v \text{ have }P_p)}{\mathbb P(v\text{ has } P_p)}\ge 1 - 2\exp(-c_2 n).
\end{align*}

Thus, the vertex $v$ satisfies the property from the statement of the corollary with probability at least 
\begin{equation*}
    1 - \sum_{u \text{ has label }\ell_v} \mathbb P(u \text{ does not have }P_p\hspace{0.2em}|\hspace{0.2em} v\text{ has } P_p)\ge 1 - \Theta(n^2)\exp(-c_2 n) = 1 - \exp(-(c_2+o(1)) n).
\end{equation*}
The corollary is proved.
\end{proof}

With the help of Corollary~\ref{AKM cor 2 bar}, we are ready to improve on Lemma~\ref{AKS lem 2} by showing that every vertex of sufficiently high degree in $G$ has, with high probability, many neighbours in $G$, which participate in connected components of $G_p$ of order $\Omega(n^2)$. Denote $c'_1 = \min(c_1(\varepsilon/2), 1)$ and $c'_2 = c_2(\varepsilon/2)$.

\begin{lemma}\label{AKS lem 3}
Fix $\varepsilon\in (0,0.1)$ and $p = (1+\varepsilon)/\bard$. There are constants $c_3 = c_3(\varepsilon) > 0$ and $c_4 = c_4(\varepsilon) > 0$ such that for every vertex $v$ satisfying $\deg_G(v)\ge (1-\varepsilon/32) \bard$, the following property holds with probability at least $1 - \exp(-(c_4+o(1))n)$: $v$ is adjacent (in $G$) to at least $c_3 n$ vertices, participating in connected components of $G_p$ of order at least $c_3\varepsilon n^2/32 C$.
\end{lemma}
\begin{proof}
We use the technique of two-round exposure with $p_1 = (1 + \varepsilon/2)/\bard$ and $p_2$ given by the equation $(1-p_1)(1-p_2) = (1-p)$. Since $\bard\to +\infty$ with $n$, $p_2 = (\varepsilon/2+o(1))/\bard$, so for every large enough $n$ we have $p_2\ge \varepsilon/4\bard$. 

Fix a vertex $v$ satisfying $\deg_G(v)\ge (1-\varepsilon/32)\bard$. By Corollary~\ref{AKM cor 2 bar}, applied with $\varepsilon/2$ instead of $\varepsilon$, we get that with probability $1 - \exp(-(c'_2+o(1))n)$ the vertex $v$ is a neighbour (in $G$) to at least $c'_1\varepsilon n/128 C$ disjoint cells of $G_{p_1}$, each containing at least $\varepsilon n/34 C$ vertices with the property $P_{p_1}$. We condition on this event. Fix any such vertex $v$ and let $\mathcal C_1, \dots, \mathcal C_k$ be the cells, which correspond to $v$ in the above statement, with $k\ge c'_1\varepsilon n/128 C$.

Fix an arbitrary cell, say $\mathcal C_1$, and let $u_1, u_2, \dots, u_m$ be vertices in $\mathcal C_1$ which satisfy the property $P_{p_1}$, where $m = \lfloor \varepsilon n/34 C\rfloor$. Moreover, assume that for every vertex $u_i$ and every cell $\mathcal C$ in a fixed set of $\lceil c'_1\varepsilon n/128 C\rceil$ disjoint cells, which witnesses that $u_i$ satisfies the property $P_{p_1}$, $\mathcal C$ contains exactly $\lceil \varepsilon n/16 C\rceil$ vertices (clearly any connected graph $H$ contains a connected subgraph of any order between $1$ and $|V(H)|$). For every $i\in [m]$, we associate the above set of cells to the vertex $u_i$.

Now, we consider the independent percolation of the edges in $G$ with parameter $p_2$. This is our second round. We do the following exploration process. List $u_1, \dots, u_m$ in this order and start exploring their neighbourhoods one by one. If $u_1$ connects (during the second round percolation with parameter $p_2$) to a neighbouring cell of size $\lceil \varepsilon n/16 C\rceil$, which was associated to it, then name this cell $\mathcal C'_1$. Then, go to $u_2$. If $\mathcal C'_1$ was well defined, there are at most two cells $\mathcal C'$ among the ones, associated to $u_2$, such that $|\mathcal C'\cap \mathcal C'_1|\ge \lceil \varepsilon n/16 C\rceil/2$. Let $(w^2_j)_{j\in [2]}$ be the two vertices, which connect $u_2$ to neighbouring cells, associated to $u_2$ and with the largest intersection with $\mathcal C'_1$. Then, if $u_2$ connects to a neighbouring cell associated to it via an edge, different from $u_2w^2_1$ and $u_2w^2_2$, name this cell $\mathcal C'_2$ and go to $u_3$. Then, since $|\mathcal C'_1\cup \mathcal C'_2|\le 2\lceil \varepsilon n/16 C\rceil$, there are at most $4$ cells $\mathcal C'$, associated to $u_3$, for which $|\mathcal C'\cap (\mathcal C'_1\cup \mathcal C'_2)|\ge \lceil \varepsilon n/16 C\rceil/2$. Let $(w^3_j)_{j\in [4]}$ be the four vertices, which connect $u_3$ to neighbouring cells, associated to $u_3$ and with the largest possible intersection with $\mathcal C'_1\cup \mathcal C'_2$. Then, if $u_3$ connects to a neighbouring cell associated to it via an edge, different from $(u_3w^3_j)_{j\in [4]}$, name this cell $\mathcal C'_3$ and continue with $u_4$, ets.

Suppose that in the moment of exploring the neighbourhood of the vertex $u_i$ we came across the cells $\mathcal C'_{i_1}, \mathcal C'_{i_2}, \dots, \mathcal C'_{i_j}$ for some $1\le i_1 < \dots < i_j \le i-1$. Then, 
\begin{equation*}
    \bigg| \bigcup_{1\le s\le j} \mathcal C'_{i_s} \bigg| = \sum_{1\le s\le j} \bigg|\mathcal C'_{i_s}\setminus \left(\mathcal C'_{i_1}\cup \dots\cup \mathcal C'_{i_{s-1}}\right)\bigg|\ge \dfrac{j\varepsilon n}{32 C}.
\end{equation*}
Also, for every $i\le i_0 := \lfloor c'_1\varepsilon n/512 C\rfloor$ (by definition of $c'_1$ we have $i_0\le m$), there are at least $c'_1\varepsilon n/128 C - 2 c'_1\varepsilon n/512 C\ge c'_1\varepsilon n/256 C$ cells, associated to the vertex $u_i$, which do not intersect the union of cells $\mathcal C'_{i_1}, \mathcal C'_{i_2}, \dots, \mathcal C'_{i_j}$ in more than $\lceil \varepsilon n/16 C\rceil/2$ vertices. Thus, for every large enough $n$ and every $i\le i_0$, $u_i$ has probability at least $p_3 := p_2 c'_1\varepsilon n/256C\ge c'_1\varepsilon^2/1024 C^2$ to connect to a cell, which does not intersect the union of $\mathcal C'_{i_1}, \mathcal C'_{i_2}, \dots, \mathcal C'_{i_j}$ in more than $\lceil \varepsilon n/16 C\rceil/2$ vertices. We conclude that the indicator functions of the events 
\begin{equation*}
    \left(\left\{u_i \text{ connects to a cell } \mathcal C' \text{ associated to it and such that } |\mathcal C'\cap \left(\cup_{1\le \ell\le j} \mathcal C_{i_{\ell}}\right)| < \dfrac{\lceil \varepsilon n/16 C\rceil}{2}\right\}\right)_{i\in [i_0]}
\end{equation*}
stochastically dominate a family of i.i.d. random variables $(B_i)_{i\in [i_0]}$ with Bernoulli distribution with parameter $p_3$. By a direct application of Chernoff's inequality (Lemma~\ref{chernoff}) we conclude that for every large enough $n$ and $c_3 = c_3(\varepsilon) = (c'_1)^2\varepsilon^3/2^{21} C^3$
\begin{equation*}
    \mathbb P\left(\sum_{i\in [i_0]} B_i\le \frac{p_3i_0}{2}\right)\le \exp\left(-\frac{p_3i_0}{8}\right)\le \exp\left(-\dfrac{(c'_1\varepsilon n/2)\cdot (c'_1 \varepsilon^2)}{8\cdot 512C\cdot 1024 C^2}\right) = \exp(-c_3 n/4).
\end{equation*}

Thus, for every large enough $n$ and every $\ell\in [k]$, the connected component of $\mathcal C_{\ell}$ in $G_p = G_{p_1}\cup G_{p_2}$ contains at least $(p_3 i_0/2)\cdot (\varepsilon n/32 C)\ge c_3\varepsilon n^2/32 C$ vertices with probability at least $1-\exp(-c_3 n/4)$. A union bound over all $k\le Cn$ cells shows that, for every large enough $n$ and for every vertex $v$ satisfying $\deg_G(v)\ge (1-\varepsilon/32) \bard$ and with at least $c'_1\varepsilon n/128 C$ neighbouring cells of order $\lceil\varepsilon n/16 C\rceil$ in $G_{p_1}$, with probability at least $1 - Cn\exp(-c_3 n/4)$ each of the connected components of $\mathcal C_1, \mathcal C_2, \dots, \mathcal C_k$ in $G_p$ is of order at least $c_3\varepsilon n^2/32 C$.

Now, let $c_4 = \min(c'_2, c_3/4)$. Then, with probability at least $1 - \exp(-(c_4+o(1)) n)$, the vertex $v$ is incident to at least $c_3 n$ vertices, which participate in connected components of $G_p$ of size at least $c_3\varepsilon n^2/32 C$. The lemma is proved.
\end{proof}

Up to this moment, we ensured the existence of a large number of connected components of order at least $\Theta(n^2)$ in $G_p$. Recall that $\gamma$ is a positive constant such that, for every $j\in [n]$, $n^{-\gamma}\le i(G_j)$. If $\gamma\in (0,1)$, we are ready to complete the proof of Theorem~\ref{main thm}. However, for larger values we will need to show more. In the sequel we ensure that there are a lot of components of size $\Omega(n^{\gamma'+2})$ in $G_p$ for some $\gamma' > \gamma$. The aim of the next lemma is to iterate the procedure of Lemma~\ref{AKS lem 2}, Corollary~\ref{AKM cor 2 bar} and Lemma~\ref{AKS lem 3} to provide the a.a.s. existence of these larger connected components. Unlike the results we presented above, we will mostly rely on the asymptotic notations $\Theta$ and $\Omega$ in the proof of Lemma~\ref{AKS lemma plus} rather than give explicit constants to simplify the presentation, having in mind that very similar but more precise formulations of the claims below were already presented in detail.

\begin{lemma}\label{AKS lemma plus}
Fix any integer $k\ge 2$, any $\varepsilon\in (0, 0.1)$ and $p = (1+\varepsilon)/\bard$. Then, there are positive constants $\beta_k\ge 32, C_k = C_k(\varepsilon), C'_k = C'_k(\varepsilon), C''_k = C''_k(\varepsilon)$ such that for every vertex $v$ satisfying $\deg_G(v)\ge (1-\varepsilon/\beta_k) \bard$, the following property holds with probability at least $1 - \exp(-(C''_k+o(1))n)$: $v$ is adjacent (in $G$) to at least $C'_k n$ vertices, participating in connected components in $G_p$ of order at least $C_k n^k$.
\end{lemma}
\begin{proof}
We argue by induction. By Lemma~\ref{AKS lem 3} the statement is true for $k=2$ with parameters $\beta_2 = 32, C_2 = c_3\varepsilon/32C, C'_2 = c_3$ and $C''_2 = c_4$ for every $\varepsilon\in (0, 0.1)$.

Suppose that the statement is satisfied for some $k-1\ge 2$. 
Fix $p_0 = (1+\varepsilon/4)/\bard$, $p'_0 = (\varepsilon/4+o(1))/\bard$ and $p_1 = (1+\varepsilon/2)/\bard$ so that $(1-p_0)(1-p'_0) = 1 - p_1$. Moreover, for any vertex $v$ in $G$, denote by $P_{p,k}$ the following property:

\begin{center}
the vertex $v$ is a neighbour in $G$ of $\Omega(n)$ disjoint cells in $G_p$, each of order $\Omega(n^k)$.
\end{center}

The next claim is an analogue of Lemma~\ref{AKS lem 2}, so we give only the main points of the proof.

\begin{claim}\label{AKS claim 2}
Every vertex of degree at least $(1-\varepsilon/8\beta_{k-1})\bard$ in $G$ has property $P_{p_1, k-1}$ with probability $1 - \exp(-\Omega(n))$.
\end{claim}
\begin{proof}
We follow the proof of Lemma~\ref{AKS lem 2}. Fix a vertex $v = (v_1, v_2, \dots, v_n)$ satisfying $\deg_G(v)\ge (1-\varepsilon/8\beta_{k-1})\bard$. For every $i\in [n]$, let $u_i$ be a neighbour of $v_i$ in $G_i$. Denote $\hat{v}_i = (v_1, \dots, v_{i-1}, u_i, v_{i+1}, \dots, v_n)$ and 
\begin{equation*}
    H_i := \left(\underset{j\in [i-1]}{\square} v_j\right)\square\hspace{0.2em} u_i\hspace{0.2em}\square\left(\underset{k\in [n]\setminus [i]}{\square} G_k\right).
\end{equation*}

Then, for every $\varepsilon\in (0, 0.1)$ there exists a positive constant $\overline{c}_{k-1} = \overline{c}_{k-1}(\varepsilon)\le 1/2$ such that, for every $i\le \overline{c}_{k-1}n$, the degree of $\hat{v}_i$ in $H_i$ is at least $(1-\varepsilon/4\beta_{k-1})\bard(H_i) = (1-(\varepsilon/4)/\beta_{k-1})\bard(H_i)$. By the induction hypothesis, applied with $\varepsilon/4$, $p_0$, $\hat{v}_i$ and $H_i$ (which is isomorphic to the product of at least $n-i\ge n/2$ of the graphs $(G_j)_{j\in [n]}$), the vertex $\hat{v}_i$ is incident to at least $\Theta(n/2)$ vertices, participating in connected components in $H_{i,p_0}$ of order $\Omega((n/2)^{k-1})$ with probability $1 - \exp(-\Omega(n/2))$. (Note that although $\Omega(n) = \Omega(n/2)$ and $\Omega(n^{k-1}) = \Omega((n/2)^{k-1})$, we add the constants to indicate that the graph $H_i$ is a product of less that $n$, but at least $n/2$ graphs. When considered appropriate, similar implicit indications are given below as well).

It remains to notice that the graphs $(H_i)_{1\le i\le \overline{c}_{k-1}n}$ are disjoint and therefore the vertices $(\hat{v}_i)_{1\le i\le \overline{c}_{k-1}n}$ connect to cells of order $\Omega((n/2)^{k-1})$ in $(H_i)_{1\le i\le \overline{c}_{k-1}n}$ respectively at the second round percolation with parameter $p'_0$ independently and with probability $p'_0 \Omega(n/2) = \Omega(1)$. Thus, by Chernoff's inequality (Lemma~\ref{chernoff}) we deduce that the vertex $v$ has $\Omega(1) \cdot \overline{c}_{k-1} n/2 = \Omega(n)$ neighbours, which participate into disjoint cells of $G_{p_1} = G_{p_0}\cup G_{p'_0}$ of size $\Omega((n/2)^{k-1})$ with probability $1 - \exp(-\Omega(1) \cdot \overline{c}_{k-1} n/8) = 1 - \exp(-\Omega(n))$. The proof is completed.
\end{proof}

\begin{claim}\label{AKS claim 2 bar}
Every vertex $v$ with $\deg_G(v)\ge (1-\varepsilon/16\beta_{k-1})\bard$ satisfies the following property with probability $1 - \exp(-\Omega(n))$: the vertex $v$ is a neighbour in $G$ to $\Omega(n)$ disjoint cells of $G_{p_1}$, each containing $\Omega(n)$ vertices with the property $P_{p_1, k-1}$.
\end{claim}
\begin{proof}
We follow the proof of Corollary~\ref{AKM cor 2 bar}. 
Fix any vertex $v$ satisfying $\deg_G(v)\ge (1-\varepsilon/16\beta_{k-1})\bard$. Since $\beta_{k-1}\ge 32$, by Lemma~\ref{AKS lem 2} $v$ has probability $1-\exp(-\Omega(n))$ to have property $P_{p_1}$. We condition on this event. Then, for every cell $\mathcal C$ among the $\Omega(n)$ disjoint neighbouring cells of order at least $\lceil\varepsilon n/16C\rceil$, corresponding to $v$, put a label $\ell_v$ on the $\lceil\varepsilon n/(16 C\beta_{k-1} + 1)\rceil$ vertices of $\mathcal C$ that are closest to $v$ in the graph $G_{p_1}$ (if some set of vertices is at the same distance to $v$ in $G_p$, make an arbitrary choice which of them to label, if necessary). Thus, for every vertex $u$ which has received a label $\ell_v$ we have $d_G(u,v)\le d_{G_{p_1}}(u,v)\le \lceil\varepsilon n/(16 C\beta_{k-1} + 1)\rceil$. Moreover, by Corollary~\ref{cor neighbours} for every large enough $n$ we have
$$|\deg(u) - \deg(v)|\le C\lceil\varepsilon n/(16 C\beta_{k-1} + 1)\rceil\le \varepsilon n/16 \beta_{k-1}$$ 
and so 
$$\deg(u)\ge \deg(v) - \varepsilon n/16 \beta_{k-1}\ge \bard - \varepsilon \bard/16 \beta_{k-1} - \varepsilon n/16 \beta_{k-1} \ge (1 - \varepsilon/8 \beta_{k-1}) \bard.$$

Note that a total of $O(n^2)$ vertices will receive the label $\ell_v$, and furthermore by Claim~\ref{AKS claim 2} each of these vertices has property $P_{p_1, k-1}$ with probability $1 - \exp(-\Omega(n))$. Then, conditionally on the event that $v$ has property $P_{p_1}$, any vertex $u$ with label $\ell_v$ has property $P_{p_1, k-1}$ with probability
\begin{align*}
    & \mathbb P(u \text{ has }P_{p_1, k-1}\hspace{0.2em}|\hspace{0.2em} v\text{ has } P_{p_1}) = \dfrac{\mathbb P(u \text{ has } P_{p_1, k-1} \text{ and } v \text{ has }P_{p_1})}{\mathbb P(v\text{ has } P_{p_1})}\ge 1 - 2\exp(-\Omega(n)) = 1 - \exp(-\Omega(n)).
\end{align*}

Thus, the vertex $v$ satisfies the property from the statement of the claim with probability at least 
\begin{equation*}
    1 - \sum_{u \text{ has label }\ell_v} \mathbb P(u \text{ does not have }P_{p_1, k-1}\hspace{0.2em}|\hspace{0.2em} v\text{ has } P_{p_1})\ge 1 - O(n^2)\exp(-\Omega(n)) = 1 - \exp(-\Omega(n)).
\end{equation*}
The claim is proved.
\end{proof}

The finish the proof of the lemma, we follow the ideas of the proof of Lemma~\ref{AKS lem 3}. We use once again the technique of two-round exposure with $p_1 = (1 + \varepsilon/2)/\bard$ and $p_2 = (\varepsilon/2+o(1))/\bard$ such that $(1-p_1)(1-p_2) = (1-p)$.

Fix a vertex $v$ satisfying $\deg_G(v)\ge (1-\varepsilon/32 \beta_{k-1})\bard$. By Claim~\ref{AKS claim 2 bar} we get that, for some positive constant $\hat{C}_{k-1} = \hat{C}_{k-1}(\varepsilon)$, with probability $1 - \exp(-\Omega(n))$ the vertex $v$ is a neighbour (in $G$) to at least $\hat{C}_{k-1} n$ disjoint cells of $G_{p_1}$, each containing $\Omega(n)$ vertices with the property $P_{p_1, k-1}$. Let us condition on this event, and let $\mathcal C_1, \dots, \mathcal C_t$ be the cells, which correspond to $v$ in the above statement, where $t$ is an integer satisfying $\hat{C}_{k-1} n\le t\le Cn$.

Fix an arbitrary cell among $\mathcal C_1, \dots, \mathcal C_t$, say $\mathcal C_1$, and let $u_1, u_2, \dots, u_m$ be vertices in $\mathcal C_1$ which satisfy the property $P_{p_1, k-1}$, where $m = \Omega(n)$. Moreover, assume that for every vertex $u_i$ and every cell $\mathcal C$ among a fixed set of $\Omega(n)$ disjoint cells, which witnesses that $u_i$ satisfies the property $P_{p_1, k-1}$, $\mathcal C$ contains exactly $\lceil \hat{C}'_{k-1} n^{k-1}\rceil$ vertices, where $\hat{C}'_{k-1}$ is a positive constant depending only on $k$ and $\varepsilon$. By the very same exploration procedure as in the proof of Lemma~\ref{AKS lem 3} we show that with probability $1 - \exp(-\Omega(n))$ the connected component of the cell $\mathcal C_1$ in $G_p = G_{p_1}\cup G_{p_2}$ contains at least $s = \Omega(n)$ cells $\mathcal C'_1, \mathcal C'_2, \dots, \mathcal C'_s$ such that, for every $i\in [s]$, 
\begin{equation*}
    |V(\mathcal C'_i\setminus \cup_{j\in [i-1]} \mathcal C'_j)|\ge \dfrac{\lceil \hat{C}'_{k-1} n^{k-1}\rceil}{2}.
\end{equation*}
Recall that $\mathcal C'_1, \mathcal C'_2, \dots, \mathcal C'_s$ are cells, associated to different vertices among $u_1, u_2, \dots, u_m$, and therefore these are not necessarily disjoint. Thus, bounding from below the number of vertices in every new cell that does not participate in the union of the previous ones is crucial to attain the lower bound on the size of the union.

Since the above reasoning applies to each of the cells $\mathcal C_1, \mathcal C_2, \dots, \mathcal C_t$, associated to $v$, and $t = \Theta(n)$, we conclude by union bound that $v$ satisfies property $P_{p,k}$ with probability $1 - \Theta(n)\exp(-\Omega(n)) = 1 - \exp(-\Omega(n))$, which finishes the proof (the constants $C_k, C'_k$ and $C''_k$ are hidden in the $\Omega$ notation but $\beta_k$ could be defined recursively by $\beta_k = 32\beta_{k-1}$, so in particular $\beta_k = 32^{k-1}$).
\end{proof}

We are ready to prove Theorem~\ref{main thm}~b). Fix $k = \lceil 1+\gamma\rceil + 3$. Up to now, we have ensured the existence of a number of cells in $G_p$, which contain at least $C_k n^k$ vertices. In the sequel, let $\hat{C}_k = C_k(\varepsilon/2), \hat{C}'_k = C'_k(\varepsilon/2)$ and $\hat{C}''_k = C''_k(\varepsilon/2)$.

\begin{proof}[Proof of Theorem~\ref{main thm}~b)]
It is sufficient to prove the claim for every $\varepsilon\in (0, 0.1)$. Once again, we consider two-round exposure of $G_p = G_{p_1}\cup G_{p_2}$ with $p_1 = (1+\varepsilon/2)/\bard$ and $p_2 = (\varepsilon/2+o(1))/\bard$. By Lemma~\ref{AKS lemma plus} and Markov's inequality with probability at least $1 - \exp(-(\hat{C}''_k/2+o(1))n)$ all but at most an $\exp(-\hat{C}''_k n/2)$-proportion of all vertices of degree at least $(1-(\varepsilon/2)/\beta_k)\bard$ in $G$ have at least $\hat{C}'_k n$ neighbours, which participate in connected components of $G_{p_1}$ of order at least $\hat{C}_k n^k$. We condition on this event. Then, the number of edges, adjacent to vertices in connected components of $G_{p_1}$ of order at least $\hat{C}_k n^k$ is at least $(1+o(1)) \hat{C}'_k n |V(G)|/2$ (recall that asymptotically almost all vertices of $G$ have degree at least $(1-(\varepsilon/2)/\beta_k)\bard$). Since every vertex has degree at most $C n$ in $G$, there are at least $(1+o(1)) \hat{C}'_k |V(G)|/2C$ vertices of $G$ in connected components of $G_{p_1}$ of order at least $\hat{C}_k n^k$.

We prove that the following property holds with probability $1 - \exp(-\Omega(|V(G)|/n^k))$:\\

\noindent
the vertices in all connected components in $G_{p_1}$ of order at least $\hat{C}_k n^k$ cannot be partitioned into two sets, $V_1$ and $V_2$, such that $||V_1|-|V_2||\le (|V_1|+|V_2|)/3$ (or equivalently $|V_1|/2\le |V_2|\le |V_1|$ up to symmetry considerations) so that there is no path in $G_p$ between $V_1$ and $V_2$.\\

\noindent
On the above event, by Observation~\ref{ob discrete cont} we may directly conclude that the largest connected component of $G_p$ contains at least $(|V_1|+|V_2|)/3\ge (1+o(1)) \hat{C}_k |V(G)|/6 C$ vertices, which would finish the proof of Theorem~\ref{main thm}~b).

Since the number of connected components of $G_{p_1}$ of order at least $\hat{C}_k n^k$ is $O(|V(G)|/n^k)$, there are $2^{O(|V(G)|/n^k)}$ ways to partition these components into two sets. We will be interested only in partitions $(V_1, V_2)$ such that $|V_1|/2\le |V_2|\le |V_1|$. Consider two cases:
\begin{enumerate}
    \item $N_{G}[V_1]\cap N_{G}[V_2]\ge |V(G)|/n^{k-2}$, and
    \item $N_{G}[V_1]\cap N_{G}[V_2] < |V(G)|/n^{k-2}$.
\end{enumerate}

In the first case, we know by our conditioning that $(1 + o(1))|V(G)|/n^{k-2}$ of the vertices in $N_{G}[V_1]\cap N_{G}[V_2]$ have at least $\hat{C}'_k n/2$ neighbours (in $G$) in either $V_1$ or $V_2$, or in both. Therefore, the probability that a fixed vertex in $N_{G}[V_1]\cap N_{G}[V_2]$ connects $V_1$ and $V_2$ at the second round percolation with parameter $p_2$ is at least $(1 - (1-p_2)^{\hat{C}'_k n/2})\cdot p_2 = \Theta(1/n)$. Moreover, the above events are independent for different vertices in $N_{G}[V_1]\cap N_{G}[V_2]$. Therefore, the probability that $V_1$ and $V_2$ do not get connected at the second round percolation with parameter $p_2$ is at most
\begin{equation*}
    \left(1-\Theta(1/n)\right)^{(1+o(1))|V(G)|/n^{k-2}} =  \exp\left(-\Theta(|V(G)|/n^{k-1})\right) \ll 2^{-O(|V(G)|/n^k)}.
\end{equation*}

In the second case, since the number of edges between $V_2$ and $N_G(V_2)$ is at least $i(G) |V_2|$ by assumption, by Corollary~\ref{isoper cor} there are at least $i(G) |V_2|/Cn \ge n^{-1-\gamma} |V_2|/2C$ vertices in $V(G)\setminus V_2$, adjacent to $V_2$ in $G$. But 
$$|N_{G}[V_1]\cap N_{G}[V_2]| < |V(G)|/n^{k-2}\ll n^{-1-\gamma} |V_2|/2C,$$ 
so by our conditioning each of the $(1+o(1)) n^{-\gamma-1} |V_2|/2C$ vertices in $N_G[V_2]\setminus N_G[V_1]$ have at least $\hat{C}'_k n$ edges towards $V_2$ in $G$. On the other hand, since $|V_1|\ge |V_2|$ and $|V_1\cap N_G[V_2]| = o(|V_1|)$, we have by Corollary~\ref{isoper cor} that there are at least 
$$i(G) \min(|N_G[V_2]\setminus N_G[V_1]|, (1+o(1))|V_1|) = \Omega(n^{-\gamma} |V(G)|)$$ 
edges, going out of $N_G[V_2]\setminus N_G[V_1]$. One may directly deduce that there are $\Omega(n^{-\gamma} |V(G)|)/Cn = \Omega(n^{-\gamma-1} |V(G)|)$ disjoint edges, which have one endvertex in $N_G[V_2]\setminus N_G[V_1]$ and one endvertex in $V(G)\setminus N_G[V_2]$. Since all but $\exp(-\Omega(n)) |V(G)|$ vertices have at least $\hat{C}'_k n$ edges towards $V_1\cup V_2$ by our conditioning, we deduce that there are $\Omega(n^{-\gamma - 1} |V(G)|)$ disjoint edges $uv$ in $G$ such that $u$ has at least $\hat{C}'_k n$ edges towards $V_1$ and $v$ has at least $\hat{C}'_k n$ edges towards $V_2$. We conclude that for any such edge $u$ and $v$ there is a path from $V_1$ through $u$ and $v$ towards $V_2$ with probability $(1 - (1-p_2)^{\hat{C}'_k n})\cdot p_2\cdot (1 - (1-p_2)^{\hat{C}'_k n}) = \Theta(1/n)$. Therefore, the probability that $V_1$ and $V_2$ do not get connected at the second round percolation with parameter $p_2$ is at most
\begin{equation*}
    \left(1-\Theta(1/n)\right)^{\Omega(|V(G)|/n^{\gamma+1})} =  \exp\left(-\Omega(|V(G)|/ n^{\gamma+2})\right) \ll 2^{-O(|V(G)|/n^k)}.
\end{equation*}
We conclude the proof of Theorem~\ref{main thm}~b) by a union bound over all $2^{O(|V(G)|/n^k)}$ partitions of the components of size at least $\hat{C}_k n^k$ in $G_{p_1}$.
\end{proof}

\section{Discussion and further questions}\label{sec discussion}
In this paper we proved that there is a sharp threshold for the existence of a giant component after percolation of the product graph $G = G_1\square \dots\square G_n$ under the assumptions that $\max_{j\in [n]} \Delta(G_j)$ is uniformly bounded from above by a constant and $\min_{j\in [n]} i(G_j)$ decays to zero at most polynomially fast. As Remark~\ref{rem 2.1} points out, at the price of a more technical exposition Theorem~\ref{main thm} may be generalised for graphs with slowly increasing degrees. Except for simplicity, we spared the details also because we believe that Theorem~\ref{main thm} may also be proved in an even more general setting.

To begin with, we were not able to find convincing counterexamples of the sharp threshold phenomenon without the maximum degrees assumption. In the proof of Theorem~\ref{main thm} presented above, this assumption was used in most of our lemmas.

\begin{question}
Can one prove an analogue of Theorem~\ref{main thm} without the assumption on the maximum degrees of $(G_j)_{j\in [n]}$?
\end{question}

Concerning the assumption on the decay of the isoperimetric constants, we show that it cannot be removed entirely. Consider the graph $G$ where $G_1 = G_2 = \dots = G_{n-1}$, each containing two vertices ($0$ and $1$) a single edge ($01$), and $G_n$ being a cycle of length $2^{2^n}$. Then, all vertices in $G$ will have degree $n+1$. Fix $p = 2/(n+1)$. Note that for any edge $uv$ of $G_n$ we have that the probability that each of the edges $((x,u)(x,v))_{x\in \{0,1\}^{n-1}}$ of $G$ disappears after $p$-percolation is $(1 - 2/(n+1))^{2^{n-1}} = \exp(-(1+o(1)) 2^n/(n+1))$. Thus, on average many of the sets of edges $((x,u)(x,v))_{x\in \{0,1\}^{n-1}}\subseteq E(G)$ for different edges $uv$ of $G_n$ disappear a.a.s. after $p$-percolation, so no giant component exists since for any two edges $u_1v_1$ and $u_2v_2$ of $G_n$, the edge set $\{(x,u_1)(x,v_1)\}_{x\in \{0,1\}^{n-1}}\cup \{(x,u_2)(x,v_2)\}_{x\in \{0,1\}^{n-1}}\subseteq E(G)$ forms a cut in $G$. Although somewhat trivial, this example leads to another logical question.

\begin{question}
Can one prove an analogue of Theorem~\ref{main thm} if $\min_{j\in [n]}i(G_j)$ decreases faster than a polynomial function of $n$?
\end{question}

Of course, graph products other than the Cartesian product exist as well. It might be interesting to study the appearance of a giant component with respect to them.

\begin{question}
Can one prove analogous results for other graph products?
\end{question}

\section{Acknowledgements}
The author would like to thank Dieter Mitsche, Guillem Perarnau and Ivailo Hartarsky for several useful remarks, and Felix Joos for turning my attention to the reference~\cite{Joo2}. I am also grateful to the two anonymous referees for a number of important comments and suggestions.

\bibliographystyle{plain}
\bibliography{References}

\begin{thebibliography}{10}

\bibitem{AKS}
M.~Ajtai, J.~Koml{\'o}s, and E.~Szemer{\'e}di.
\newblock Largest random component of a $k$-cube.
\newblock {\em Combinatorica}, 2(1):1--7, 1982.

\bibitem{Ald}
D.~Aldous.
\newblock Brownian excursions, critical random graphs and the multiplicative
  coalescent.
\newblock {\em The Annals of Probability}, pages 812--854, 1997.

\bibitem{ABC}
N.~Alon, I.~Benjamini, and A.~Stacey.
\newblock Percolation on finite graphs and isoperimetric inequalities.
\newblock {\em The Annals of Probability}, 32(3):1727--1745, 2004.

\bibitem{Bla}
B.~Blaszczyszyn.
\newblock Lecture notes on random geometric models - random graphs, point
  processes and stochastic geometry.
\newblock Lecture notes, December 2017.

\bibitem{Bol1}
B.~Bollob{\'a}s.
\newblock The evolution of random graphs.
\newblock {\em Transactions of the American Mathematical Society},
  286(1):257--274, 1984.

\bibitem{BKL}
B.~Bollob{\'a}s, Y.~Kohayakawa, and T.~{\L}uczak.
\newblock The evolution of random subgraphs of the cube.
\newblock {\em Random Structures \& Algorithms}, 3(1):55--90, 1992.

\bibitem{BCVdHSS1}
C.~Borgs, J.~T. Chayes, R.~van~der Hofstad, G.~Slade, and J.~Spencer.
\newblock Random subgraphs of finite graphs: {I}. {T}he scaling window under
  the triangle condition.
\newblock {\em Random Structures \& Algorithms}, 27(2):137--184, 2005.

\bibitem{BCVdHSS2}
C.~Borgs, J.~T. Chayes, R.~van~der Hofstad, G.~Slade, and J.~Spencer.
\newblock Random subgraphs of finite graphs: {II}. {T}he lace expansion and the
  triangle condition.
\newblock {\em The Annals of Probability}, 33(5):1886--1944, 2005.

\bibitem{BCVdHSS3}
C.~Borgs, J.~T. Chayes, R.~van~der Hofstad, G.~Slade, and J.~Spencer.
\newblock Random subgraphs of finite graphs: {III}. {T}he phase transition for
  the n-cube.
\newblock {\em Combinatorica}, 26(4):395--410, 2006.

\bibitem{CHL}
F.~Chung, P.~Horn, and L.~Lu.
\newblock Percolation in general graphs.
\newblock {\em Internet Mathematics}, 6(3):331--347, 2009.

\bibitem{CT}
F.~R.~K. Chung and P.~Tetali.
\newblock Isoperimetric inequalities for cartesian products of graphs.
\newblock {\em Combinatorics Probability and Computing}, 7(2):141--148, 1998.

\bibitem{Chu}
R.~K. Chung.
\newblock {CBMS}, spectral graph theory.
\newblock {\em American Mathematical Society}, 1994.

\bibitem{ER2}
P.~Erd\H{o}s and A.~Rényi.
\newblock On random graphs {I}.
\newblock {\em Publ. math. Debrecen}, 6(290-297):18, 1959.

\bibitem{ER1}
P.~Erd{\H{o}}s and A.~R{\'e}nyi.
\newblock On the evolution of random graphs.
\newblock {\em Publ. Math. Inst. Hung. Acad. Sci}, 5(1):17--60, 1960.

\bibitem{ER3}
P.~Erd{\H{o}}s and A.~R{\'e}nyi.
\newblock On the strength of connectedness of a random graph.
\newblock {\em Acta Mathematica Academiae Scientiarum Hungarica},
  12(1-2):261--267, 1964.

\bibitem{ES}
P.~Erd{\"o}s and J.~Spencer.
\newblock Evolution of the n-cube.
\newblock {\em Computers \& Mathematics with Applications}, 5(1):33--39, 1979.

\bibitem{HN}
T.~Hulshof and A.~Nachmias.
\newblock Slightly subcritical hypercube percolation.
\newblock {\em Random Structures \& Algorithms}, 56(2):557--593, 2020.

\bibitem{JLR}
S.~Janson, T.~{\L}uczak, and A.~Rucinski.
\newblock {\em Random graphs}, volume~45.
\newblock John Wiley \& Sons, 2011.

\bibitem{Joo2}
F.~Joos.
\newblock Random subgraphs in sparse graphs.
\newblock {\em SIAM Journal on Discrete Mathematics}, 29(4):2350--2360, 2015.

\bibitem{Luc1}
T.~{\L}uczak.
\newblock Component behavior near the critical point of the random graph
  process.
\newblock {\em Random Structures \& Algorithms}, 1(3):287--310, 1990.

\bibitem{Mac}
C.~McDiarmid.
\newblock On the method of bounded differences.
\newblock {\em Surveys in combinatorics}, 141(1):148--188, 1989.

\bibitem{Til}
J.-P. Tillich.
\newblock Edge isoperimetric inequalities for product graphs.
\newblock {\em Discrete Mathematics}, 213(1-3):291--320, 2000.

\bibitem{vdHN}
R.~van~der Hofstad and A.~Nachmias.
\newblock Hypercube percolation.
\newblock {\em Journal of the European Mathematical Society}, 19(3):725--814,
  2017.

\end{thebibliography}

\end{document}